\documentclass[a4paper,11pt,reqno]{amsart}
\usepackage[a4paper,twoside,centering]{geometry}
\usepackage[arrow,matrix,curve]{xy}

\title{Algebras of quasi-quaternion type}
\author{Sefi Ladkani}
\address{%
Institut des Hautes \'{E}tudes Scientifiques \\
Le Bois Marie, 35, route de Chartres \\
91440 Bures-sur-Yvette, France}
\email{sefil@ihes.fr}

\DeclareMathOperator{\add}{add}
\DeclareMathOperator{\End}{End}
\DeclareMathOperator{\HH}{HH}
\DeclareMathOperator{\per}{per}
\DeclareMathOperator{\stmod}{\underline{mod}}

\newcommand{\balpha}{\bar{\alpha}}
\newcommand{\oa}{\omega_\alpha}
\newcommand{\oba}{\omega_{\balpha}}
\newcommand{\eps}{\varepsilon}

\newcommand{\cC}{\mathcal{C}}
\newcommand{\gL}{\Lambda}
\newcommand{\bZ}{\mathbb{Z}}
\newcommand{\wh}{\widehat}

\newtheorem{theorem}{Theorem}[section]
\newtheorem{prop}[theorem]{Proposition}
\newtheorem{lemma}[theorem]{Lemma}
\newtheorem{cor}[theorem]{Corollary}

\theoremstyle{definition}
\newtheorem{defn}[theorem]{Definition}

\newtheorem{example}[theorem]{Example}

\numberwithin{equation}{section}

\begin{document}

\begin{abstract}
We define algebras of quasi-quaternion type, which are symmetric algebras
of tame representation type whose stable module category has certain
structure similar to that of the algebras of quaternion type introduced
by Erdmann. We observe that symmetric tame algebras that are also
2-CY-tilted are of quasi-quaternion type.

We present a combinatorial construction of such algebras by introducing the
notion of triangulation quivers.
The class of algebras that we get contains Erdmann's algebras of quaternion
type on the one hand and the Jacobian algebras of the quivers with potentials
associated by Labardini to triangulations of closed surfaces with punctures
on the other hand,
hence it serves as a bridge between modular representation theory of finite
groups and cluster algebras.
\end{abstract}

\maketitle

\section{Introduction}

The purpose of this note is to report on some connections between
representation theory of groups and cluster algebras, more precisely,
between algebras of quaternion type introduced and studied by
Erdmann~\cite{Erdmann90} and others and 2-CY-tilted algebras which
are endomorphism algebras of cluster-tilting objects in 2-Calabi-Yau
categories arising in the additive categorification of cluster algebras.

Recall that an algebra of quaternion type is a tame, symmetric, indecomposable
algebra whose non-projective modules are $\Omega$-periodic with period
dividing 4 and its Cartan matrix is non-singular. The possible quivers
with relations of such algebras were classified by Erdmann~\cite{Erdmann90},
and later works of Holm~\cite{Holm99} and
Erdmann-Skowronski~\cite{ES06} established that the algebras
given in those lists are actually of quaternion type.

It seems natural to remove the condition that the Cartan matrix is
non-singular and to consider tame, symmetric, indecomposable algebras
whose non-projective modules are $\Omega$-periodic of period dividing 4.
In terms of the stable module category, the last condition means that the
4-th power of the suspension (shift) functor acts as the identity on
objects.
Such algebras will be called \emph{algebras of quasi-quaternion type}.

We construct a large class of algebras of quasi-quaternion type that are also
2-CY-tilted. It turns out that this class contains in particular:
\begin{itemize}
\item
All the algebras appearing in Erdmann's lists of algebras of quaternion
type~\cite{Erdmann90};

\item
All the Jacobian algebras of the quivers with potentials associated by
Labardini to triangulations of closed surfaces with punctures
\cite{Labardini09}.
\end{itemize}

Our construction has several consequences, both for the representation theory
of finite-dimensional algebras as well as for theory of quivers with
potentials. Namely, we obtain:
\begin{enumerate}
\renewcommand{\labelenumi}{\theenumi.}
\item
A new proof that the algebras in Erdmann's lists are of quaternion type;

\item
New tame symmetric algebras with periodic modules which seem not to
appear in the classification announced by Erdmann and
Skowronski~\cite{ES08};

\item
New symmetric 2-CY-tilted algebras in addition to the ones arising from
odd-dimensional isolated hypersurface singularities~\cite{BIKR08};

\item
Infinitely many non-degenerate potentials with pairwise non-isomorphic
Jacobian algebras on the adjacency quiver~\cite{FST08} of any triangulation
of a closed surface with exactly one puncture (and arbitrary genus).
\end{enumerate}

We observe that the property of being of quasi-quaternion type
is preserved under derived equivalences (see below), hence
our strategy is to construct some of these algebras from combinatorial
data and then produce more algebras using derived equivalences.
To this end we introduce \emph{triangulation quivers}.
These are quivers having the property that at any vertex there are exactly two
incoming arrows and two outgoing arrows, together with the data of a
permutation $f$ on the set of arrows such that $f(\alpha)$ starts where
an arrow $\alpha$ ends, subject to the condition that $f^3$ is the identity
(this last condition justifies the term ``triangulation'').
These data give rise to another permutation $g$ and an involution
$\alpha \mapsto \balpha$ on the set of arrows, see Section~\ref{ssec:quiver}.

A triangulation quiver can be dually encoded as a ribbon graph
whose nodes are the cycles of the permutation $g$, its edges are the
vertices of the quiver and the cyclic ordering of the edges around each node
is induced by $g$.
Thus functions on the nodes can be viewed as functions on the arrows that
are constant on $g$-cycles.
Given multiplicities and scalars
associated to the nodes, one can construct from such
data a Brauer graph algebra. We construct another algebra which we
call \emph{triangulation algebra} and prefer to work in a complete
setting; each arrow of the quiver gives rise to a certain commutativity
relation and the algebra is defined as the quotient of the complete
path algebra by the closure of the ideal generated by these commutativity
relations. \emph{A-priori} it is not clear that the triangulation algebra
is finite-dimensional, but it turns out that for most triangulation
quivers and multiplicities, the triangulation algebra
satisfies certain additional zero-relations of length 3 which allow to prove
that it is finite-dimensional.

\begin{figure}
\[
\begin{array}{lcc}
\text{\textbf{Surface}} & \text{\textbf{Ribbon graph}}
& \text{\textbf{Triangulation quiver}}
\\ \hline & & \\
\text{Monogon, unpunctured} &
\begin{array}{c}
\xymatrix{
\circ \ar@{-}@(ul,dl)_{1}
}
\end{array}
&
\begin{array}{c}
\xymatrix@=1pc{
{\bullet_1} \ar@(dl,ul)[]^{\alpha} \ar@(dr,ur)[]_{\beta}
}
\end{array}
\\ & & (\alpha) (\beta)
\\ \hline
\text{Monogon, one puncture} &
\begin{array}{c}
\xymatrix@=1.5pc{
& {\circ} \ar@{-}[r]^1
& {\circ} \ar@{-}@/_1.5pc/[ll]_(.7)2 \ar@{-}@/^1.5pc/[ll]^(.7)2
}
\end{array}
&
\begin{array}{c}
\xymatrix{
{\bullet_1} \ar@(ul,dl)[]_{\alpha} \ar@<-0.5ex>[r]_{\beta}
& {\bullet_2} \ar@(dr,ur)[]_{\eta} \ar@<-0.5ex>[l]_{\gamma}
}
\end{array}
\\ & &  (\alpha \beta \gamma) (\eta)
\\ \hline
\text{Triangle, unpunctured} &
\begin{array}{c}
\xymatrix{
{\circ} \ar@{-}@(ur,ul)[]_3 \ar@{-}@(l,dl)[]_1 \ar@{-}@(r,dr)[]^2
}
\end{array}
&
\begin{array}{c}
\xymatrix@=1pc{
& {\bullet_3} \ar@(ur,ul)[]_{\alpha_3} \ar[ddl]_{\beta_3} \\ \\
{\bullet_1} \ar@(ul,dl)[]_{\alpha_1} \ar[rr]_{\beta_1}
&& {\bullet_2} \ar@(dr,ur)[]_{\alpha_2} \ar[uul]_{\beta_2}
}
\end{array}
\\ & &
(\alpha_1) (\alpha_2) (\alpha_3) (\beta_1 \beta_2 \beta_3)
\\ \hline
\text{Sphere, three punctures} &
\begin{array}{c}
\xymatrix@=1pc{
{\circ} \ar@{-}[rr]^3 && {\circ} \ar@{-}[ddl]^2 \\ \\
& {\circ} \ar@{-}[uul]^1
}
\end{array}
&
\begin{array}{c}
\xymatrix@=1pc{
& {\bullet_3} \ar@<-0.5ex>[ddl]_{\alpha_3} \ar@<-0.5ex>[ddr]_{\beta_3} \\ \\
{\bullet_1} \ar@<-0.5ex>[rr]_{\alpha_1} \ar@<-0.5ex>[uur]_{\beta_1}
&& {\bullet_2} \ar@<-0.5ex>[ll]_{\beta_2} \ar@<-0.5ex>[uul]_{\alpha_2}
}
\end{array}
\\ & &
(\alpha_1 \alpha_2 \alpha_3) (\beta_3 \beta_2 \beta_1)
\\ &
\begin{array}{c}
\xymatrix@=1.5pc{
& {\circ} \ar@{-}[r]^1
& {\circ} \ar@{-}@/_1.5pc/[ll]_(.7)2 \ar@{-}@/^1.5pc/[ll]^(.7)2 \ar@{-}[r]^3
& {\circ}
}
\end{array}
&
\begin{array}{c}
\xymatrix{
{\bullet_1} \ar@(ul,dl)[]_{\alpha} \ar@<-0.5ex>[r]_{\beta}
& {\bullet_2} \ar@<-0.5ex>[r]_{\delta} \ar@<-0.5ex>[l]_{\gamma}
& {\bullet_3} \ar@(dr,ur)[]_{\xi} \ar@<-0.5ex>[l]_{\eta}
}
\end{array}
\\ & &
(\alpha \beta \gamma) (\delta \xi \eta)
\\ \hline
\text{Torus, one puncture} &
\begin{array}{c}
\xymatrix@=1.5pc{
& {\circ} \ar@{-}[dr]^2  \\
{\circ} \ar@{-}[ur]^1 \ar@{-}[rr]^3 && {\circ} \ar@{-}[dl]^1 \\
& {\circ} \ar@{-}[ul]^2
}
\end{array}
&
\begin{array}{c}
\xymatrix@=1pc{
& {\bullet_3} \ar@<-0.5ex>[ddl]_{\alpha_2} \ar@<0.5ex>[ddl]^(.4){\alpha_5} \\ \\
{\bullet_1} \ar@<-0.5ex>[rr]_{\alpha_0} \ar@<0.5ex>[rr]^(.4){\alpha_3} &&
{\bullet_2} \ar@<-0.5ex>[uul]_{\alpha_1} \ar@<0.5ex>[uul]^(.4){\alpha_4}
}
\end{array}
\\ & & (\alpha_4 \alpha_2 \alpha_0) (\alpha_5 \alpha_3 \alpha_1)
\\ \hline
\end{array}
\]
\caption{The triangulation quivers with at most 3 vertices.
We list the marked surface, the ribbon graph(s) corresponding to its
triangulation(s) and the associated triangulation quivers, where we write the
permutation $f$ in cycle form below each quiver. For the torus, all nodes in
the ribbon graph should be identified and edges with the same label are also
identified.}
\label{fig:quivers}
\end{figure}

Our main results concerning triangulation algebras are summarized in
the next theorem. For the precise definitions of the terms occurring in
the formulation, we refer the reader to Section~\ref{ssec:algebra}.

\begin{theorem} \label{t:quasi}
Let $(Q,f)$ be a connected triangulation quiver, let $K$ be a field, let
$m \colon Q_1 \to \bZ_{>0}$ and $c \colon Q_1 \to K^{\times}$ be
$g$-invariant functions of
multiplicities and scalars, and assume that $m$ is admissible.
Assume further that the associated ribbon graph with multiplicities is 
not one of the two exceptional cases shown in Figure~\ref{fig:except}
and consider the corresponding triangulation algebra $\gL$ defined by
\[
\gL = \wh{KQ} / 
\overline{\langle \balpha \cdot f(\balpha) -
c_\alpha \oa^{m_\alpha-1} \cdot \oa' \rangle}_{\alpha \in Q_1} .
\]
\begin{enumerate}
\renewcommand{\theenumi}{\alph{enumi}}
\item \label{it:t:finite}
$\gL$ is finite dimensional; it has a presentation as quiver with relations
\begin{equation} \label{e:qrel}
\gL \simeq KQ /
\langle \alpha \cdot f(\alpha) \cdot gf(\alpha) \,,\,
\balpha \cdot f(\balpha) - c_\alpha \oa^{m_\alpha-1} \cdot \oa' \rangle_{\alpha \in Q_1}
\end{equation}
\item \label{it:t:symmetric}
$\gL$ is symmetric.

\item \label{it:t:tame}
$\gL$ degenerates to the corresponding Brauer graph algebra $\Gamma$ given by
\[
\Gamma = KQ /
\langle \alpha \cdot f(\alpha) \,,\,
c_\alpha \oa^{m_\alpha} - c_{\balpha} \oba^{m_{\balpha}}
\rangle_{\alpha \in Q_1}
\]
and hence $\gL$ is of tame representation type.

\item \label{it:t:potential}
The elements
$\rho_\alpha = f(\alpha) f^2(\alpha) - c_\alpha \omega_{g(\alpha)}^{m_\alpha-1}
\omega_{g(\alpha)}'$ satisfy
$\sum_{\alpha \in Q_1} [\alpha, \rho_\alpha] = 0$ in $KQ$, hence
$\gL$ is a Jacobian algebra of a hyperpotential
(see~\cite{Ladkani14} for the definition)
and therefore it is 2-CY-tilted, i.e.\ there is a 2-Calabi-Yau
triangulated category $\cC$ and a cluster-tilting object $T$ in $\cC$
such that $\gL \simeq \End_{\cC}(T)$.

\item \label{it:t:quasi}
$\gL$ is of quasi-quaternion type.

\item \label{it:t:quasimut}
More generally,
for any cluster-tilting object $T'$ in $\cC$
which is reachable from $T$ by a sequence of mutations, the algebra
$\End_{\cC}(T')$
is derived equivalent to $\gL$ and of quasi-quaternion type.
\end{enumerate}
\end{theorem}

The exceptional cases are dealt with in the next proposition.
\begin{prop} \label{p:quasiex}
Let $(Q,f)$ be a connected triangulation quiver, let $K$ be a field, let
$m \colon Q_1 \to \bZ_{>0}$ and $c \colon Q_1 \to K^{\times}$ be
$g$-invariant functions of multiplicities and scalars.
Assume that the associated ribbon graph with multiplicities is one of the
two exceptional cases shown in Figure~\ref{fig:except} and that moreover:
\begin{itemize}
\item
$\prod_{\alpha \in Q_1} c_\alpha \neq 1$ in the punctured monogon case; or
\item
$c_\alpha c_{\balpha} c_{f(\alpha)} c_{f(\balpha)} \neq 1$ for some $\alpha \in Q_1$
in the tetrahedron case.
\end{itemize}
Then the statements of Theorem~\ref{t:quasi} hold for the triangulation
algebra $\gL$ with the following modifications of claims~\eqref{it:t:finite}
and~\eqref{it:t:tame}:
\begin{enumerate}
\item[$(\mathrm{a}')$]
$\gL$ is finite dimensional; it has a presentation as quiver with relations
\[
\gL \simeq KQ /
\langle \balpha \cdot f(\balpha) - c_\alpha \oa^{m_\alpha-1} \cdot \oa'
\rangle_{\alpha \in Q_1}
\]
and the zero relations $\alpha \cdot f(\alpha) \cdot gf(\alpha)$ follow from
the commutativity relations.

\item[$(\mathrm{c}')$] \label{it:p:tame}
$\gL$ is of tame representation type.
\end{enumerate}
\end{prop}

\begin{figure}
\[
\begin{array}{ccc}
\begin{array}{c}
\xymatrix{
& {\circ_3} \ar@{-}[r]
& {\circ_1} \ar@{-}@/_1.5pc/[ll] \ar@{-}@/^1.5pc/[ll]
}
\end{array}
& \qquad &
\begin{array}{c}
\xymatrix{
{\bullet} \ar@(ul,dl)[] \ar@<-0.5ex>[r]
& {\bullet} \ar@(dr,ur)[] \ar@<-0.5ex>[l]
}
\end{array}
\\ \\ \\
\begin{array}{c}
\xymatrix@=1pc{
&& {\circ_1} \ar@{-}[dddll] \ar@{-}[dddrr] \ar@{-}[dd] \\ \\
&& {\circ_1} \ar@{-}[dll] \ar@{-}[drr] \\
{\circ_1} \ar@{-}[rrrr] &&&& {\circ_1}
}
\end{array}
&&
\begin{array}{c}
\xymatrix@=0.5pc{
& {\bullet} \ar[rr] \ar@/^/[dddd]
&& {\bullet} \ar[ddr] \ar@/_3pc/[ddlll] \ar@{<-}@/^3pc/[dddd] \\ \\
{\bullet} \ar[uur] & && &
{\bullet} \ar[ddl] \ar@/^/[uulll] \ar@{<-}@/_/[ddlll] \\ \\
& {\bullet} \ar[uul]
&& {\bullet} \ar[ll] \ar@{<-}@/^3pc/[uulll]
}
\end{array}
\end{array}
\]
\caption{Exceptional ribbon graphs with multiplicities:
A monogon with one puncture (top) and a tetrahedron, which is a
triangulation of a sphere with 4 punctures (bottom).
The ribbon graph is shown on the left, where at each node we indicate
its multiplicity. The corresponding triangulation quiver is shown
on the right.}
\label{fig:except}
\end{figure}
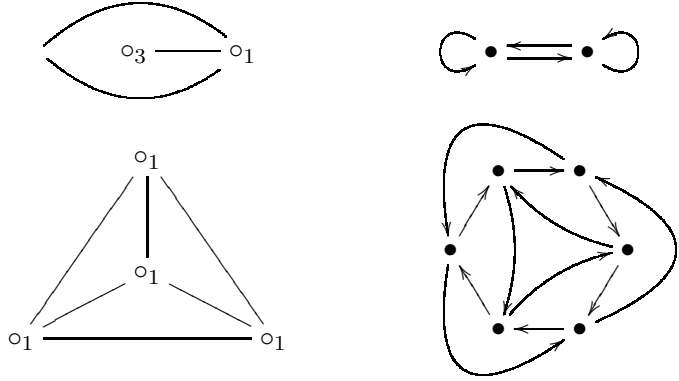

One could also formulate a slightly more general version of
Theorem~\ref{t:quasi} and Proposition~\ref{p:quasiex} by replacing the
multiplicities and the scalars by power series in $K[[x]]$, i.e.\
replace the functions $m$ and $c$ with a $g$-invariant function
$q \colon Q_1 \to K[[x]]$ and consider the algebra
\begin{equation} \label{e:q}
\gL = \wh{KQ} / 
\overline{\langle \balpha \cdot f(\balpha) -
q_\alpha(\oa) \cdot \oa' \rangle}_{\alpha \in Q_1} ,
\end{equation}
so that the case treated here corresponds to the choice of
$q_\alpha(x) = c_\alpha x^{m_\alpha-1}$.
However, in most cases the algebra $\gL$ in~\eqref{e:q} depends only on the
leading term of each power series $q_\alpha(x)$,
so for simplicity we chose not to formulate the results in full generality.
We hope to report on the general case in a later version.

Any triangulation of a marked surface in the sense of
Fomin, Shapiro and Thurston~\cite{FST08} gives rise to a triangulation
quiver (see Section~\ref{ssec:quiver})
and hence, by choosing multiplicities and scalars, to algebras
of quasi-quaternion type. Hence, as opposed to algebras of quaternion
type, there are algebras of quasi-quaternion type with arbitrarily
many non-isomorphic simple modules.

The triangulation quivers with small number of vertices can be
enumerated, see Figure~\ref{fig:quivers} for the quivers with at most
three vertices.
In particular, some algebras of quaternion type with 1, 2 or 3 vertices
arise from a monogon, a punctured monogon or a sphere with three
punctures, respectively, see Example~\ref{ex:triang}.
Some of the triangulation algebras arising from a punctured monogon or a
sphere with 3 punctures arise also from minimally elliptic curve
singularities, see Section~7 of~\cite{BIKR08}.

In general, the triangulation quiver constructed from a triangulation
differs from the adjacency quiver constructed in~\cite{FST08}.
However, for a triangulation of a closed surface satisfying
a technical condition called (T3) in our work~\cite{Ladkani12} these two
quivers coincide and the triangulation algebra (where all multiplicities
are set to~1)
coincides with the Jacobian algebra of the potential constructed
by Labardini~\cite{Labardini09}. Since any closed surface considered
in~\cite{Labardini09} admits such a triangulation and any other
triangulation can be obtained from it by a sequence of flips, by
using the facts that a flip of triangulations results in a mutation of
the corresponding quivers with potentials~\cite{Labardini09}
and that mutation of quivers with potentials is compatible with mutation
of cluster-tilting objects~\cite{BIRS11} we get the following result.

\begin{cor}
Consider a closed surface
which is not a sphere with less than $4$ punctures.
Then the Jacobian algebras arising from its triangulations are of
quasi-quaternion type and they are all derived equivalent to each other.
Moreover, they arise as algebras in part~\eqref{it:t:quasimut} of
Theorem~\ref{t:quasi} for a suitable triangulation quiver.
\end{cor}
Note that for the proof of this result one does not need to know that the
potentials are non-degenerate. Note also that for a sphere with $4$ punctures
one has to use Proposition~\ref{p:quasiex} and impose the corresponding
restriction on the scalars.

\medskip

The proof of parts~\eqref{it:t:finite} and~\eqref{it:t:symmetric}
of Theorem~\ref{t:quasi} is
similar to our proofs in the case of the Jacobian algebras of the quivers
with potentials arising from triangulations of closed
surfaces~\cite{Ladkani12}. We note that in the presentation~\eqref{e:qrel}
it is enough to require only one zero relation
$\alpha \cdot f(\alpha) \cdot gf(\alpha)$, as the rest would follow from
that relation and the commutativity relations.

In~\cite{Holm99} Holm establishes the tameness of the algebras of quaternion
type by showing that some of them degenerate to algebras of dihedral type
and then applying a result of Geiss~\cite{Geiss95}.
Part~\eqref{it:t:tame} can be seen as a generalization of this statement
to arbitrary triangulation quivers. We note that connections between
Brauer graph algebras and cluster mutations have also been discovered by
Marsh and Schroll~\cite{MarshSchroll13}. For the two exceptional cases
considered in Proposition~\ref{p:quasiex}, the statement~$(\mathrm{c}'$)
holds since the corresponding triangulation algebras are of tubular
type~\cite{BS03}.

In part~\eqref{it:t:potential}
we use the notion of a hyperpotential introduced in~\cite{Ladkani14}
in order to formulate the results in a characteristic-free form.
In particular, we get that 2-blocks with quaternion defect group are
2-CY-tilted.
The triangulation algebra $\gL$ is a Jacobian algebra of a quiver with
potential as defined in~\cite{DWZ08} when the characteristic of the ground
field $K$ is zero or does not divide any of the multiplicities $m_\alpha$.
In that case a potential can be written as
\begin{equation} \label{e:potential}
\sum_{\alpha} \alpha \cdot f(\alpha) \cdot f^2(\alpha) -
\sum_{\beta} m_{\beta}^{-1} c_{\beta} \omega_{\beta}^{m_\beta}
\end{equation}
where the sums run over representatives of $f$-cycles and $g$-cycles,
respectively.

Parts~\eqref{it:t:quasi} and~\eqref{it:t:quasimut} are consequences of
parts~\eqref{it:t:symmetric}, \eqref{it:t:tame} and~\eqref{it:t:potential}.
In fact, the statements therein hold more generally
for any tame symmetric 2-CY-tilted algebra $\gL$. Part~\eqref{it:t:quasi}
follows from the next proposition which records some observations on
symmetric algebras that are also 2-CY-tilted.
\begin{prop}
Let $\gL$ be a finite-dimensional symmetric algebra that is also
2-CY-tilted, i.e.\ $\gL = \End_{\cC}(T)$ for some cluster-tilting object $T$
within a triangulated 2-Calabi-Yau category $\cC$ with suspension functor
$\Sigma$.
\begin{enumerate}
\renewcommand{\theenumi}{\alph{enumi}}
\item \label{it:p:omega4}
The functor $\Omega^4$ on the stable module category $\stmod \gL$
is isomorphic to the identity, hence all non-projective
$\gL$-modules are $\Omega$-periodic with period dividing $4$.

\item \label{it:p:sigma2}
The functor $\Sigma^2$ acts as the identity on the objects of $\cC$.

\item
Assume that $\gL$ is a Jacobian algebra of a hyperpotential.
Then it is rigid if and only if $\gL$ is semi-simple.
\end{enumerate}
\end{prop}
Here, by \emph{rigid} we mean that $\HH_0(\gL) = \gL/[\gL,\gL]$
is spanned by the images of the primitive idempotents corresponding to the
vertices. This definition is equivalent to the one in~\cite{DWZ08} for
finite-dimensional Jacobian algebras of quivers with potentials.
Parts~\eqref{it:p:omega4} and~\eqref{it:p:sigma2} of the proposition have
also been recently observed by Valdivieso-Diaz~\cite{Valdivieso13}.

The derived equivalences in part~\eqref{it:t:quasimut} are instances of
(refined version of) good mutations introduced in our previous
work~\cite{Ladkani10}. They follow from a more general statement
concerning the derived equivalences of neighboring 2-CY-tilted algebras
which is an improvement of~\cite[Theorem~5.3]{Ladkani10}.
Before stating the theorem, we recall some relevant notions.

Let $\gL$ be a basic algebra and $P$ an indecomposable projective
$\gL$-module and write $\gL = P \oplus Q$. Consider the silting
mutations in the sense of Aihara and Iyama~\cite{AiharaIyama12}
of $\gL$ at $P$ within the triangulated category $\per \gL$
of perfect complexes, which are the following two-term complexes
\begin{align*}
U^-_P(\gL) = (P \to Q') \oplus Q &,&
U^+_P(\gL) = (Q'' \to P) \oplus Q ,
\end{align*}
where $Q', Q'' \in \add Q$, the maps are left/right
$(\add Q)$-approximations and $Q, Q', Q''$ are in degree 0.
These two-term complexes of projective modules are known also as
Okuyama-Rickard complexes. In~\cite{Ladkani10} we considered these
complexes in relation with our definition of mutations of algebras.

\begin{theorem} \label{t:dereq}
Let $T$ be a cluster-tilting object in a 2-Calabi-Yau category $\cC$, 
let $X$ be an indecomposable summand of $T$ and let $T'$ be the
cluster-tilting object which is the
Iyama-Yoshino mutation~\cite{IyamaYoshino08} of $T$ at $X$.

Consider the 2-CY-tilted algebras $\gL = \End_{\cC}(T)$ and
$\gL'= \End_{\cC}(T')$.
Let $P$ be the indecomposable projective $\gL$-module corresponding to $X$
and let $P'$ be the indecomposable projective $\gL'$-module
corresponding to $X$.
\begin{enumerate}
\renewcommand{\theenumi}{\alph{enumi}}
\item
If $U^-_P(\gL)$ and $U^+_{P'}(\gL')$ are tilting complexes, then
\[
\End_{\per \gL} U^-_P(\gL) \simeq \gL' \qquad \text{and} \qquad
\End_{\per \gL'} U^+_{P'}(\gL') \simeq \gL .
\]

\item
If $U^+_P(\gL)$ and $U^-_{P'}(\gL')$ are tilting complexes, then
\[
\End_{\per \gL} U^+_P(\gL) \simeq \gL' \qquad \text{and} \qquad
\End_{\per \gL'} U^-_{P'}(\gL') \simeq \gL .
\]

\item \label{it:algmut}
If $\gL$ is weakly symmetric, then by~\cite{HerschendIyama11}
$\gL'$ is also weakly symmetric, hence all
the complexes $U^-_P(\gL)$, $U^+_P(\gL)$, $U^-_{P'}(\gL')$
and $U^+_{P'}(\gL')$ are tilting complexes and
\[
\End_{\per \gL} U^-_P(\gL) \simeq \gL' \simeq
\End_{\per \gL} U^+_P(\gL) .
\]
In particular, $\gL$ and $\gL'$ are derived equivalent.

\item
If $\gL$ is symmetric then $\gL'$ is symmetric.
\end{enumerate}
\end{theorem}
We note that there are related works by Dugas~\cite{Dugas10}
concerning derived equivalences of symmetric algebras and by
Mizuno~\cite{Mizuno12} concerning derived equivalences of
self-injective quivers with potential.

The category of perfect complexes over a symmetric algebra is 0-Calabi-Yau,
hence the derived equivalences in part~\eqref{it:algmut} can be
considered as 0-CY analogs of the derived equivalences of
Iyama-Reiten~\cite{IyamaReiten08} and Keller-Yang~\cite{KellerYang11}
for 3-CY-algebras.

Rephrasing part~\eqref{it:algmut}, we see that if $\gL$ is a
(weakly) symmetric 2-CY-tilted algebra and $P$ an indecomposable
projective $\gL$-module, then the algebras $\End_{\per \gL} U^-_P(\gL)$
and $\End_{\per \gL} U^+_P(\gL)$ are isomorphic to each other,
2-CY-tilted and derived equivalent to~$\gL$.
A careful look at the derived equivalences constructed by
Holm~\cite{Holm99} for
algebras of quaternion type shows that all of them arise from
tilting complexes of the form appearing in part~\eqref{it:algmut}
above. Since the representatives of the derived classes are
triangulation algebras and hence 2-CY-tilted, we deduce that all the
algebras of quaternion type are of the form given in
Theorem~\ref{t:quasi}\eqref{it:t:quasimut} and in particular they are
2-CY-tilted.

Many of the algebras occurring in part~\eqref{it:t:quasimut} of
Theorem~\ref{t:quasi} are themselves triangulation algebras.
In fact, one can define a notion of mutation of triangulation quivers
that will lead to mutation of the potentials~\eqref{e:potential},
see Section~\ref{ssec:mut}.

Finally, we note that an argument as in Prop.~2.1 and Prop.~2.2
of~\cite{Holm99} yields the following observation.
\begin{prop}
\sloppy
Any algebra which is derived equivalent to an algebra of quasi-quaternion
type is also of quasi-quaternion type.
\end{prop}

\subsection*{Acknowledgements}
Some of the results reported here were obtained during
my stay at the University of Bonn and were scheduled to be presented
at the ARTA conference that was held in September 2013 at Torun, Poland.
During that period the author was supported by DFG grant LA 2732/1-1 in
the framework of the priority program SPP 1388 ``Representation theory''.

The report was written during my visit to IHES at Bures-sur-Yvette.
I would like to thank the IHES for the hospitality and the inspiring
atmosphere.
This report has been completed within the last days of my long pleasant
postdoctoral stay in Europe for over 6 years. I hope to find some time in the
future to produce a more detailed version.

I discussed various aspects of this work with Thorsten Holm, Maxim Kontsevich,
Robert Marsh and Andrzej Skowronski. I thank them for their interest.

\section{Combinatorial construction of algebras of quasi-quaternion type}
\label{sec:comb}

\subsection{Ribbon quivers and triangulation quivers}
\label{ssec:quiver}

A \emph{quiver} is a finite directed graph. More precisely, it is a quadruple
$Q=(Q_0,Q_1,s,t)$ where $Q_0$ and $Q_1$ are finite sets (of \emph{vertices}
and \emph{arrows}, respectively) and $s,t \colon Q_1 \to Q_0$ are functions
specifying for each arrow its starting and terminating
vertex, respectively.

\begin{defn}
A \emph{ribbon quiver} is a pair $(Q,f)$ consisting of a quiver
$Q$ and a permutation $f \colon Q_1 \to Q_1$ on its set of arrows
satisfying the following conditions:
\begin{enumerate}
\renewcommand{\theenumi}{\roman{enumi}}
\item \label{it:ribbon:deg2}
At each vertex $i \in Q_0$ there are exactly two arrows starting at $i$ and
two arrows ending at $i$;

\item \label{it:ribbon:f}
For each arrow $\alpha \in Q_1$, the arrow $f(\alpha)$ starts where $\alpha$
ends.
\end{enumerate}
Note that loops are allowed in $Q$. A loop at a vertex is counted both
as an incoming and outgoing arrow at that vertex.
\end{defn}

Let $(Q,f)$ be a ribbon quiver.
Since at each vertex of $Q$ there are exactly two outgoing arrows,
there is an involution $\alpha \mapsto \balpha$ on $Q_1$ mapping
each arrow $\alpha$ to the other arrow starting at the vertex $s(\alpha)$.
Composing it with $f$ gives rise to the permutation
$g \colon Q_1 \to Q_1$ given by $g(\alpha) = \overline{f(\alpha)}$
so that for each arrow $\alpha$,
the set $\{f(\alpha),g(\alpha)\}$ consists of
the two arrows starting at the vertex which $\alpha$ ends at.

Given a quiver $Q$ satisfying condition~\eqref{it:ribbon:deg2} in the
definition, the data of the permutation $f$ is equivalent to the data of the
permutation $g$. Thus from now on when considering a ribbon quiver $(Q,f)$
we will freely refer to the involution $\alpha \mapsto \balpha$ and
the permutation $g$ as defined above.

Ribbon quivers are closely related to ribbon graphs.
Informally speaking, a ribbon graph is a graph consisting of nodes and
edges together with a cyclic ordering of the edges around each node.
This can be made more formal in the next definition.

\begin{defn}
A \emph{ribbon graph} is a triple $(H, \iota, \sigma)$ where $H$ is a finite
set, $\iota$ is an involution on $H$ without fixed points and
$\sigma$ is a permutation on $H$.
\end{defn}

The elements of $H$ are called \emph{half edges}.
A ribbon graph gives rise to a graph $(V,E)$ (possibly with loops and
multiple edges between nodes) as follows.
The set $E$ of edges consists of the cycles of $\iota$ and
the set $V$ of nodes consists of the cycles of $\sigma$.
An edge $e \in E$ can be written as $(h \, \iota(h))$ for some $h \in H$.
The $\sigma$-cycles that $h$ and $\iota(h)$ belong to are the nodes that
$e$ is incident to.
Finally, the cyclic ordering around each node is induced by $\sigma$.

\begin{prop}
The notions of ribbon quiver and ribbon graph are equivalent.
\end{prop}
\begin{proof}
A ribbon quiver $(Q,f)$ gives rise to a ribbon graph
$(H,\iota,\sigma)$ by taking $H=Q_1$ and defining $\iota(\alpha)=\balpha$
and $\sigma(\alpha) = \overline{f(\alpha)}$ for each $\alpha \in Q_1$.

Conversely, a ribbon graph $(H,\iota,\sigma)$ gives rise to a ribbon quiver
$(Q,f)$ as follows.
Set $Q_1=H$ and take $Q_0$ to be the set of cycles of $\iota$. Define
the maps $s,t \colon Q_1 \to Q_0$ and the permutation $f \colon Q_1 \to Q_1$
by letting, for $h \in H$, $s(h)$ to be the $\iota$-cycle that $h$ belongs to
and setting $t = s \sigma$ and $f = \iota \sigma$.

We finally note that these two constructions are inverses of each other.
\end{proof}

We will focus on a subclass of ribbon quivers formed by what we call
triangulation quivers. 

\begin{defn}
A \emph{triangulation quiver} is a ribbon quiver $(Q,f)$ such that
$f^3$ is the identity on the set of arrows.
\end{defn}

As their name suggests, triangulation quivers naturally arise from
triangulations of marked surfaces.
Following Fomin, Shapiro and Thurston~\cite{FST08}, 
a \emph{marked surface} is a pair $(S,M)$ consisting of a compact,
connected, oriented, Riemann surface $S$ (possibly with boundary)
and a finite set $M$ of points in $S$, called \emph{marked points},
such that each connected component of the boundary
of $S$ contains at least one point from $M$. The points in $M$
which are not on the boundary of $S$ are called \emph{punctures}.

We refer to~\cite{FST08} for the notion of (ideal) \emph{triangulation}
of a marked surface.

\begin{prop}
A triangulation of a marked surface gives rise to a ribbon graph
whose associated ribbon quiver is a triangulation quiver.
\end{prop}
\begin{proof}
Consider a triangulation $\tau$ of a marked surface $(S,M)$.
We associate to $\tau$ a ribbon graph as follows:
the nodes are the punctures in $M$ and the connected components of the
boundary of $S$, and the edges are the arcs of $\tau$ as well as the boundary
segments (sides of triangles which are part of the boundary).

For each boundary segment on a boundary component $C$
we draw the corresponding edge as a loop incident to the node 
corresponding to $C$.
In this way each marked point $p$ on $C$ could be identified with the
``space'' between the consecutive loops corresponding to the two boundary
segments which have $p$ as endpoint, see an example in Figure~\ref{fig:node}.

Using this identification, we can now draw the edges corresponding to arcs,
placing them correctly between the loops (if an endpoint of the arc is on 
a boundary).
The cyclic ordering at each node is the counterclockwise ordering induced by
the orientation of $S$.

\begin{figure}
\[
\xymatrix{
{\circ} \ar@{-}@(ur,ul)[] \ar@{-}@(ul,dl)[]
\ar@{-}@(dl,dr)[] \ar@{-}@(dr,ur)[]
}
\]
\caption{A node with $4$ loops corresponds to a boundary component with
$4$ marked points.}
\label{fig:node}
\end{figure}
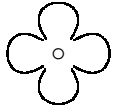

The vertices of the corresponding ribbon quiver are the arcs of $\tau$
as well as the boundary segments.
At each vertex corresponding to a boundary segment where is a loop
$\delta$ with $f(\delta)=\delta$, and each triangle in $\tau$ 
with sides $v_1, v_2, v_3$ (which may be arcs or
boundary segments) arranged in a clockwise order gives rise to
three arrows
$v_1 \xrightarrow{\alpha} v_2$, $v_2 \xrightarrow{\beta} v_3$
and $v_3 \xrightarrow{\gamma} v_1$ with
$f(\alpha)=\beta$, $f(\beta)=\gamma$ and $f(\gamma)=\alpha$.
\end{proof}

The construction of the triangulation quiver of an ideal triangulation
resembles that of the adjacency quiver defined in~\cite{FST08},
however there are several differences:
\begin{enumerate}
\renewcommand{\labelenumi}{\theenumi.}
\item
In the triangulation quiver there are vertices corresponding to the
boundary segments and not only to the arcs.

\item
Our treatment of self-folded triangles is different; in the triangulation
quiver there is a loop at each vertex corresponding to the inner side
of a self-folded triangle.

\item
We do not delete 2-cycles that arise in the quiver (e.g.\ when there
are precisely two arcs incident to a puncture).
\end{enumerate}

These differences allow to attach triangulation quivers to 
marked surfaces that do not admit adjacency quivers, such as
a monogon, a triangle or a sphere with three punctures, see
Figure~\ref{fig:quivers}. On the other hand, there are situations where
the triangulation quiver and the adjacency quiver of a triangulation
coincide.

\begin{lemma} \label{l:adjacency}
The triangulation quiver equals the adjacency quiver for any
triangulation of a closed surface (i.e.\ with empty boundary)
with at least three arcs incident to each puncture.
\end{lemma}
The condition in the lemma was called (T3) in our work~\cite{Ladkani12}.
In particular, we get the following corollary.

\begin{cor} \label{c:onep}
For a closed surface with exactly one puncture, the triangulation quiver
and the adjacency quiver associated to any triangulation coincide.
\end{cor}

\subsection{Brauer graph algebras and triangulation algebras}
\label{ssec:algebra}

\begin{defn}
Let $(Q,f)$ be a ribbon quiver. A function
$\nu \colon \alpha \mapsto \nu_\alpha$ on $Q_1$ is called
\emph{$g$-invariant} if $\nu_{g(\alpha)} = \nu_\alpha$ for
any arrow $\alpha$.

A $g$-invariant function can thus be regarded as a function on the
nodes of the associated ribbon graph.
\end{defn}

Let $(Q,f)$ be a ribbon quiver. For an arrow $\alpha \in Q_1$, set
\begin{align*}
n_\alpha &= \min \{n>0 \,:\, g^n(\alpha)=\alpha\} \\
\oa &= \alpha \cdot g(\alpha) \cdot \ldots \cdot g^{n_\alpha-1}(\alpha) \\
\oa' &= \alpha \cdot g(\alpha) \cdot \ldots \cdot g^{n_\alpha-2}(\alpha)
\end{align*}

The function $\alpha \mapsto n_\alpha$ is obviously $g$-invariant, telling the
length of the $g$-cycle $\oa$ starting at $\alpha$. The path $\oa'$ is
``almost'' cycle; when $n_\alpha=1$ the arrow $\alpha$ is a loop at some vertex
$i$ and $\oa'$ is understood to be the path of length zero starting at $i$.

Let $K$ be a field. For a quiver $Q$, denote by $KQ$ its path algebra
over $K$ and by $\wh{KQ}$ the completed path algebra. The elements
of $KQ$ are finite $K$-linear combinations of paths in $Q$ whereas
those of $\wh{KQ}$ are possibly infinite such combinations.

\begin{defn}
Let $(Q,f)$ be a ribbon quiver, and let $m \colon Q_1 \to \bZ_{>0}$
and $c \colon Q_1 \to K^{\times}$ be $g$-invariant functions of
\emph{multiplicities} and \emph{scalars}, respectively.
The \emph{graph algebra} associated to these data is defined by
\[
\Gamma(Q,f,m,c) = KQ /
\langle \alpha \cdot f(\alpha) \,,\,
c_\alpha \oa^{m_\alpha} - c_{\balpha} \oba^{m_{\balpha}}
\rangle_{\alpha \in Q_1}.
\]
\end{defn}
In other words, the graph algebra is the Brauer graph algebra~\cite{Kauer98}
associated to the corresponding ribbon graph. In particular, it is special
biserial and hence of tame representation type.

\begin{defn}
Let $(Q,f)$ be a triangulation quiver and let $m \colon Q_1 \to \bZ_{>0}$
and $c \colon Q_1 \to K^{\times}$ be $g$-invariant functions of
\emph{multiplicities} and \emph{scalars}, respectively.

We say that $m$ is \emph{admissible} if $m_\alpha n_\alpha \geq 3$ for
every arrow $\alpha \in Q_1$. In this case we define the
\emph{triangulation algebra} associated to these data as a quotient
of the completed path algebra of $Q$ by the closure of an ideal
generated by suitable commutativity relations:
\[
\gL(Q,f,m,c) = \wh{KQ} / 
\overline{\langle \balpha \cdot f(\balpha) -
c_\alpha \oa^{m_\alpha-1} \cdot \oa' \rangle}_{\alpha \in Q_1}
\]
\end{defn}

Since the path $\oa^{m_\alpha-1} \cdot \oa'$ is of length $m_\alpha n_\alpha - 1$,
the definition of a triangulation algebra makes sense also when
$m_\alpha n_\alpha=2$, but then the corresponding arrow could be eliminated
from $Q$ complicating somewhat the remaining relations.
The admissibility condition ensures that the generating relations lie
in the square of the ideal generated by all arrows of $Q$ so no arrows
have to be deleted.

\begin{example} \label{ex:triang}
We identify some algebras in the literature as triangulation algebras.
In the first three examples, we use the presentation as quiver with relations
given in Theorem~\ref{t:quasi}\eqref{it:t:finite}.
\begin{enumerate}
\renewcommand{\labelenumi}{\theenumi.}
\item
The triangulation algebras of the triangulation quiver corresponding to
a monogon are algebras of quaternion type with one vertex (notation III.1(e)
in~\cite{Erdmann90}).

\item
The triangulation algebras of the triangulation quiver corresponding to
a punctured monogon are algebras of quaternion type with two vertices
(denoted $Q(2\mathcal{B})_1$ in~\cite{Erdmann90}).

\item
The triangulation algebras of the triangulation quivers corresponding to
triangulations of a sphere with three punctures are algebras of quaternion
type with three vertices (denoted $Q(3\mathcal{D})$ and $Q(3\mathcal{K})$
in~\cite{Erdmann90}).

\item
As shown in~\cite{Ladkani12}, the Jacobian algebra of the quiver
with potential associated by Labardini-Fragoso~\cite{Labardini09}
to a triangulation of a closed surface satisfying condition (T3)
is the triangulation algebra of its adjacency quiver (which
is a triangulation quiver in view of Lemma~\ref{l:adjacency})
with all multiplicities set to 1.
\end{enumerate}
\end{example}

The reason for the exclusion of the two exceptional cases from
Theorem~\ref{t:quasi} is explained by the next statement.

\begin{prop}
Let $(Q,f)$ be a connected triangulation quiver and
$m \colon Q_1 \to \bZ_{>0}$ an admissible $g$-invariant function of
multiplicities. Then the following conditions are equivalent:
\begin{enumerate}
\renewcommand{\theenumi}{\alph{enumi}}
\item
The ribbon graph of $(Q,f)$ with multiplicities is one of the two
shown in Figure~\ref{fig:except}, i.e.\ a punctured monogon with
multiplicities $(3,1)$ or a tetrahedron with all multiplicities
equal to 1.

\item
$m_\alpha n_\alpha = 3$ for all $\alpha \in Q_1$.

\item
$(m_\alpha n_\alpha)^{-1} + (m_{f(\alpha)}n_{f(\alpha)})^{-1} + 
(m_{f^2(\alpha)} n_{f^2(\alpha)})^{-1} = 1$ for some $\alpha \in Q_1$.
\end{enumerate}
\end{prop}

\subsection{Mutations of triangulation quivers}
\label{ssec:mut}

Motivated by the relation between flips of triangulations and Fomin-Zelevinsky
mutation of their adjacency quivers~\cite{FST08}, we introduce a notion of
mutation for triangulation quivers.

\begin{defn} \label{def:mut}
Let $(Q,f)$ be a triangulation quiver and let $k$ be a vertex of $Q$
without loops. Denote by $\alpha$, $\balpha$ the two arrows that start at $k$
and observe that our assumption on $k$ implies that there are six distinct arrows
\begin{align*}
\alpha_1 = \alpha &,&
\beta_1 = f(\alpha) &,&
\gamma_1 = f^2(\alpha) &,&
\alpha_2 = \balpha &,&
\beta_2 = f(\balpha) &,&
\gamma_2 = f^2(\balpha)
\end{align*}
which form two cycles of the permutation $f$.

The \emph{mutation} of $(Q,f)$ at $k$ is the triangulation quiver $(Q',f')$
obtained from $(Q,f)$ by performing the following steps:
\begin{enumerate}
\item
Remove the two arrows $\beta_1$ and $\beta_2$;

\item
Replace the four arrows $\alpha_1$, $\alpha_2$, $\gamma_1$ and $\gamma_2$ 
with arrows in the opposite direction $\alpha^*_1$, $\alpha^*_2$,
$\gamma^*_1$ and $\gamma^*_2$;

\item
Add new arrows $\delta_{12}$ and $\delta_{21}$ with
\begin{align*}
s(\delta_{12}) = s(\gamma_1) &,&
t(\delta_{12}) = t(\alpha_2) &,&
s(\delta_{21}) = s(\gamma_2) &,&
t(\delta_{21}) = t(\alpha_1) .
\end{align*}

\item
Define the permutation $f'$ on the new set of arrows $Q'_1$ 
by $f'(\eps)=f(\eps)$ if $\eps$ is an arrow of $Q$ which has not been
changed, and by
\begin{align*}
f'(\alpha^*_1) = \gamma^*_2 &,&
f'(\gamma^*_2) = \delta_{21} &,&
f'(\delta_{21}) = \alpha^*_1 \\
f'(\alpha^*_2) = \gamma^*_1 &,&
f'(\gamma^*_1) = \delta_{12} &,&
f'(\delta_{12}) = \alpha^*_2
\end{align*}
for the other arrows.
\end{enumerate}
\end{defn}
At the level of the underlying quivers, this is very similar to Fomin-Zelevinsky
mutation, but note that $Q'$ may contain 2-cycles.
 
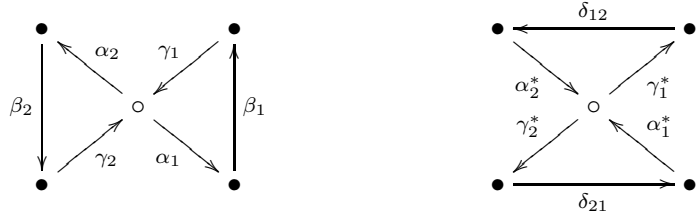
\begin{figure}
\begin{align*}
\xymatrix@R=1.5pc{
{\bullet} \ar[dd]_{\beta_2} & & {\bullet} \ar[dl]_{\gamma_1} \\
& {\circ} \ar[dr]_{\alpha_1} \ar[ul]_{\alpha_2} \\
{\bullet} \ar[ur]_{\gamma_2} & & {\bullet} \ar[uu]_{\beta_1}
}
&& 
\xymatrix@R=1.5pc{
{\bullet} \ar[dr]_{\alpha^*_2} && {\bullet} \ar[ll]_{\delta_{12}} \\
& {\circ} \ar[ur]_{\gamma^*_1} \ar[dl]_{\gamma^*_2} \\
{\bullet} \ar[rr]_{\delta_{21}} && {\bullet} \ar[ul]_{\alpha^*_1}
}
\end{align*}
\caption{Mutation of triangulation quivers at the middle vertex $\circ$.
Some of the other vertices may coincide, and only the arrows that
change are shown.}
\label{fig:mut}
\end{figure}

\begin{lemma}
The triangulation quivers of two triangulations related by a flip at some arc
are related by a mutation at the vertex corresponding to that arc.
\end{lemma}
\begin{proof}
We need to verify that a vertex corresponding to a flippable arc cannot
have loops. Indeed, for a loop $\alpha$ at some vertex $k$ we have that either
$f(\alpha)=\alpha$ or $g(\alpha)=\alpha$. In the former case $k$ corresponds
to a boundary segment, whereas in the latter case it corresponds to an arc
which is the inner side of a self-folded triangle.
\end{proof}

The permutation $f'$ on $Q'_1$ defines the permutation $g'$ by
$g'(\alpha')=\overline{f'(\alpha')}$ for $\alpha' \in Q'_1$.
Any $g$-invariant function $\nu$ gives rise to a $g'$-invariant function
$\nu'$ on $Q'_1$ by setting $\nu'_\eps = \nu_\eps$ for the arrows in $Q'_1$
that are also in $Q_1$ and
\begin{align*}
\nu'_{\alpha^*_1} = \nu'_{\gamma^*_1} = \nu_{\beta_1} &,&
\nu'_{\alpha^*_2} = \nu'_{\gamma^*_2} = \nu_{\beta_2} &,&
\nu'_{\delta_{12}} = \nu_{\gamma_1} &,&
\nu'_{\delta_{21}} = \nu_{\gamma_2}
\end{align*}
for the other arrows.
In particular, any two $g$-invariant functions $m \colon Q_1 \to \bZ_{>0}$ and
$c \colon Q_1 \to K^{\times}$ of multiplicities and scalars on $(Q,f)$
give rise to $g'$-invariant functions of multiplicities
$m' \colon Q'_1 \to \bZ_{>0}$ and scalars $c' \colon Q'_1 \to K^{\times}$
on $(Q',f')$.

For the rest of this section we fix a triangulation quiver $(Q,f)$ and
consider its mutation $(Q',f')$ at some vertex $k$ without loops.

\begin{prop}
The ribbon graphs of $(Q,f)$ and $(Q',f')$ are related by an elementary move
in the sense of Kauer~\cite{Kauer98}. Hence the corresponding Brauer graph
algebras $\Gamma(Q,f,m,c)$ and $\Gamma(Q',f',m',c')$
are derived equivalent for any choice of multiplicities and scalars.
\end{prop}

Let $p \colon Q_1 \to xK[[x]]$ be a $g$-invariant function whose values are
power series without constant term. Consider the potential on $Q$ defined by
\begin{equation} \label{e:pot_p}
W = \sum_\alpha \alpha \cdot f(\alpha) \cdot f^2(\alpha) -
\sum_\beta p_{\beta}(\omega_\beta)
\end{equation}
where the sums run over representatives $\alpha$ of $f$-cycles and $\beta$
of $g$-cycles in $Q_1$.
The function $p$ gives rise to a $g'$-invariant function $p'$
and hence to the potential on $Q'$
\[
W' = \sum_{\alpha'} \alpha' \cdot f'(\alpha') \cdot f'^2(\alpha') -
\sum_{\beta'} p'_{\beta'}(\omega_{\beta'})
\]
where the sums run over representatives $\alpha'$ of $f'$-cycles and
$\beta'$ of $g'$-cycles in $Q'_1$.

The next proposition compares $(Q',W')$ with the mutation of the quiver
with potential $(Q,W)$ at the vertex $k$ as defined in~\cite{DWZ08}.

\begin{prop} \label{p:QPmut}
Assume that there are no $2$-cycles in $Q$ passing through the vertex~$k$.
Then $(Q',W')$ is right equivalent to the mutation of $(Q,W)$ at $k$.
\end{prop}

In the notation of Definition~\ref{def:mut}, the condition in the
proposition is equivalent to the conditions that
$n_{\alpha_1} > 2$, $n_{\gamma_1} > 2$, $n_{\beta_1} > 1$ and
$n_{\beta_2} > 1$.

Combining Proposition~\ref{p:QPmut} with Corollary~\ref{c:onep}, we get:
\begin{cor}
Let $Q$ be the adjacency quiver of a triangulation of a closed surface with
exactly one puncture and view it as a triangulation quiver $(Q,f)$.
Then for any power series $p(x) \in xK[[x]]$ the potential
\[
-p(\omega) + \sum_\alpha \alpha \cdot f(\alpha) \cdot f^2(\alpha)
\]
(where the sum runs over representatives $\alpha$ of $f$-cycles and 
$\omega$ is the cycle $\omega_\beta$ for some $\beta \in Q_1$)
is non-degenerate.
In particular, the set of power series
\[ 
\{0 \} \cup \{x^m \,:\, \text{$m$ is not divisible by the characteristic
of $K$}\}
\]
yields infinitely many non-degenerate potentials on $Q$ whose Jacobian algebras
are pairwise non-isomorphic.
\end{cor}

Triangulation algebras are Jacobian algebras of quivers with potentials
under some conditions on the characteristic of the ground field.
\begin{lemma}
Let $m \colon Q_1 \to \bZ_{>0}$ and $c \colon Q_1 \to K^{\times}$ be
$g$-invariant functions and assume that all the multiplicities
$m_\alpha$ are invertible over $K$. Then the triangulation algebra
$\gL(Q,f,m,c)$ is the Jacobian algebra of $(Q,W)$ where the potential
$W$ is of the form~\eqref{e:pot_p} for the $g$-invariant
function $p \colon Q_1 \to xK[[x]]$ defined by
$p_\alpha(x) = c_\alpha m_{\alpha}^{-1} x^{m_\alpha}$.
\end{lemma}

By using the compatibility between mutations of quivers with potentials
and mutations of cluster-tilting objects~\cite{BIRS11}, noting that
the vanishing condition needed in~\cite[Theorem~5.2]{BIRS11}
is always satisfied for symmetric
(even for self-injective) algebras, Proposition~\ref{p:QPmut}
together with Theorem~\ref{t:dereq} imply the following derived
equivalence. As we work with quivers with potentials, we have to
impose some restrictions on the characteristic of the ground field.

\begin{cor}
Assume that there are no $2$-cycles in $Q$ passing through the vertex~$k$.
Let $m \colon Q_1 \to \bZ_{>0}$ and $c \colon Q_1 \to K^{\times}$ be
$g$-invariant functions of multiplicities and scalars, respectively.
Let $m'$ and $c'$ be the corresponding
$g'$-invariant functions on $Q'_1$.
Assume that $m$ is admissible and that moreover each of the numbers
$m_\alpha$ is not divisible by the characteristic of $K$.
Then the triangulation algebras $\gL(Q,f,m,c)$ and $\gL(Q',f',m',c')$
are derived equivalent.
\end{cor}

In fact, under the conditions of the corollary we have, in the notations 
of Theorem~\ref{t:dereq},
\begin{align*}
&\End_{\per \gL} U^-_{P_k}(\gL) \simeq \gL' \simeq
\End_{\per \gL} U^+_{P_k}(\gL) \\
&\End_{\per \Gamma} U^-_{P_k}(\Gamma) \simeq \Gamma' \simeq
\End_{\per \Gamma} U^+_{P_k}(\Gamma)
\end{align*}
where
\begin{align*}
\gL = \gL(Q,f,m,c) &,& \Gamma = \Gamma(Q,f,m,c) &,&
\gL' = \gL(Q',f',m',c') &,& \Gamma' = \Gamma(Q',f',m',c')
\end{align*}
and $P_k$ denotes the indecomposable projective module corresponding to the
vertex $k$ over the appropriate algebra.

\bibliographystyle{amsplain}
\bibliography{quasi}

\providecommand{\bysame}{\leavevmode\hbox to3em{\hrulefill}\thinspace}
\providecommand{\MR}{\relax\ifhmode\unskip\space\fi MR }
\providecommand{\MRhref}[2]{%
  \href{http://www.ams.org/mathscinet-getitem?mr=#1}{#2}
}
\providecommand{\href}[2]{#2}
\begin{thebibliography}{10}

\bibitem{AiharaIyama12}
Takuma Aihara and Osamu Iyama, \emph{Silting mutation in triangulated
  categories}, J. Lond. Math. Soc. (2) \textbf{85} (2012), no.~3, 633--668.

\bibitem{BS03}
Jerzy Bia{\l}kowski and Andrzej Skowro{\'n}ski, \emph{On tame weakly symmetric
  algebras having only periodic modules}, Arch. Math. (Basel) \textbf{81}
  (2003), no.~2, 142--154.

\bibitem{BIRS11}
A.~B. Buan, O.~Iyama, I.~Reiten, and D.~Smith, \emph{Mutation of
  cluster-tilting objects and potentials}, Amer. J. Math. \textbf{133} (2011),
  no.~4, 835--887.

\bibitem{BIKR08}
Igor Burban, Osamu Iyama, Bernhard Keller, and Idun Reiten, \emph{Cluster
  tilting for one-dimensional hypersurface singularities}, Adv. Math.
  \textbf{217} (2008), no.~6, 2443--2484.

\bibitem{DWZ08}
Harm Derksen, Jerzy Weyman, and Andrei Zelevinsky, \emph{Quivers with
  potentials and their representations. {I}. {M}utations}, Selecta Math. (N.S.)
  \textbf{14} (2008), no.~1, 59--119.

\bibitem{Dugas10}
Alex Dugas, \emph{A construction of derived equivalent pairs of symmetric
  algebras}, \texttt{arXiv:1005.5152}.

\bibitem{Erdmann90}
Karin Erdmann, \emph{Blocks of tame representation type and related algebras},
  Lecture Notes in Mathematics, vol. 1428, Springer-Verlag, Berlin, 1990.

\bibitem{ES06}
Karin Erdmann and Andrzej Skowro{\'n}ski, \emph{The stable {C}alabi-{Y}au
  dimension of tame symmetric algebras}, J. Math. Soc. Japan \textbf{58}
  (2006), no.~1, 97--128.

\bibitem{ES08}
\bysame, \emph{Periodic algebras}, Trends in representation theory of algebras
  and related topics, EMS Ser. Congr. Rep., Eur. Math. Soc., Z\"urich, 2008,
  pp.~201--251.

\bibitem{FST08}
Sergey Fomin, Michael Shapiro, and Dylan Thurston, \emph{Cluster algebras and
  triangulated surfaces. {I}. {C}luster complexes}, Acta Math. \textbf{201}
  (2008), no.~1, 83--146.

\bibitem{Geiss95}
Christof Geiss, \emph{On degenerations of tame and wild algebras}, Arch. Math.
  (Basel) \textbf{64} (1995), no.~1, 11--16.

\bibitem{HerschendIyama11}
Martin Herschend and Osamu Iyama, \emph{Selfinjective quivers with potential
  and 2-representation-finite algebras}, Compos. Math. \textbf{147} (2011),
  no.~6, 1885--1920.

\bibitem{Holm99}
Thorsten Holm, \emph{Derived equivalence classification of algebras of
  dihedral, semidihedral, and quaternion type}, J. Algebra \textbf{211} (1999),
  no.~1, 159--205.

\bibitem{IyamaReiten08}
Osamu Iyama and Idun Reiten, \emph{Fomin-{Z}elevinsky mutation and tilting
  modules over {C}alabi-{Y}au algebras}, Amer. J. Math. \textbf{130} (2008),
  no.~4, 1087--1149.

\bibitem{IyamaYoshino08}
Osamu Iyama and Yuji Yoshino, \emph{Mutation in triangulated categories and
  rigid {C}ohen-{M}acaulay modules}, Invent. Math. \textbf{172} (2008), no.~1,
  117--168.

\bibitem{Kauer98}
Michael Kauer, \emph{Derived equivalence of graph algebras}, Trends in the
  representation theory of finite-dimensional algebras ({S}eattle, {WA}, 1997),
  Contemp. Math., vol. 229, Amer. Math. Soc., Providence, RI, 1998,
  pp.~201--213.

\bibitem{KellerYang11}
Bernhard Keller and Dong Yang, \emph{Derived equivalences from mutations of
  quivers with potential}, Adv. Math. \textbf{226} (2011), no.~3, 2118--2168.

\bibitem{Labardini09}
Daniel Labardini-Fragoso, \emph{Quivers with potentials associated to
  triangulated surfaces}, Proc. Lond. Math. Soc. (3) \textbf{98} (2009), no.~3,
  797--839.

\bibitem{Ladkani14}
Sefi Ladkani, \emph{2-{CY}-tilted algebras that are not {J}acobian},
  \texttt{arXiv:1403.6814}.

\bibitem{Ladkani12}
\bysame, \emph{On {J}acobian algebras from closed surfaces},
  \texttt{arXiv:1207.3778}.

\bibitem{Ladkani10}
\bysame, \emph{Perverse equivalences, {BB}-tilting, mutations and
  applications}, \texttt{arXiv:1001.4765}.

\bibitem{MarshSchroll13}
Robert Marsh and Sibylle Schroll, \emph{The geometry of {B}rauer graph algebras
  and cluster mutations}, \texttt{arXiv:1309.4239}.

\bibitem{Mizuno12}
Yuya Mizuno, \emph{On mutations of selfinjective quivers with potential},
  \texttt{arXiv:1210.3166}.

\bibitem{Valdivieso13}
Yadira Valdivieso-Diaz, \emph{On {A}uslander-{R}eiten translation in cluster
  categories associated to closed surfaces}, \texttt{arXiv:1309.2708}.

\end{thebibliography}

\end{document}